\documentclass[preprint]{imsart}

\usepackage[T1]{fontenc}
\usepackage[latin1]{inputenc}
\usepackage[english]{babel}
\usepackage{amsmath,amsthm,amssymb,amsfonts,epic,latexsym,hyperref
}
\usepackage{mathabx}
\usepackage[dvips]{graphicx}
\usepackage[usenames]{color}

\usepackage{dsfont}

\usepackage{lipsum}
\usepackage{rotating}
\usepackage{dsfont}
\usepackage{geometry}
\usepackage{stmaryrd}
\usepackage{ifthen}
\usepackage[nottoc]{tocbibind}
\usepackage{verbatim}
\usepackage{tabularx}
\usepackage{framed}
\usepackage{longtable}
\usepackage{tabu}
\usepackage{mathtools}

\allowdisplaybreaks

\def\cP{{\mathcal P}}

\def \E{I\!\!E}

\def\N{{\bf N}}

\def \cC{\mathcal C}

\def \cP{\mathcal P}

\newcommand{\R}{\mathbb {R}}

\def\indiq{{\bf 1}}

\newcommand{\B}{\mathcal{B}}
\newcommand{\F}{\mathcal{F}}

\newcounter{comptetape}
\newtheorem{thm}{Theorem}[section]
\newtheorem{prop}[thm]{Proposition}
\newtheorem{cor}[thm]{Corollary}
\newtheorem{ass}{Assumption}
\newtheorem{lem}[thm]{Lemma}
\newtheorem{rem}[thm]{Remark}

\begin{document}

\begin{frontmatter}

\title{Mean field limits of interacting particle systems with positive stable jumps}

\runtitle{Systems of particles with stable jumps}

\begin{aug}

\author{\fnms{Eva} \snm{L\"ocherbach}\thanksref{m1}\ead[label=e4]{locherbach70@gmail.com}},
\author{\fnms{Dasha}
\snm{Loukianova}\thanksref{m2}\ead[label=e5]{dasha.loukianova@univ-evry.fr}}

\address{\thanksmark{m1}Statistique, Analyse et Mod\'elisation Multidisciplinaire, Universit\'e Paris 1 Panth\'eon-Sorbonne, EA 4543 et FR FP2M 2036 CNRS
  \thanksmark{m2}Laboratoire de Math\'ematiques et Mod\'elisation d'\'Evry,  Universit\'e
  d'\'Evry Val d'Essonne, UMR CNRS 8071 
  }

\runauthor{E. L\"ocherbach and D. Loukianova}

 \end{aug}

\begin{abstract}
This note is a companion article to the recent paper \cite{DEE}. We consider mean field systems of interacting particles. Each particle jumps with a jump rate depending on its position. When jumping,  a macroscopic quantity is added to its own position. Moreover, simultaneously, all other particles of the system receive a small random kick which is distributed according to a positive $\alpha-$stable law and scaled in $N^{-1/\alpha},$ where $0 <  \alpha < 1.$ In between successive jumps of the system, the particles follow a deterministic flow with drift depending on their position and on the empirical measure of the total system. In a more general framework where jumps and state space do not need to be positive, we have shown in \cite{DEE} that the mean field limit of this system is a McKean-Vlasov type process which is solution of a non-linear SDE, driven by an $ \alpha-$stable process. Moreover we have obtained in \cite{DEE} an upper bound for the strong rate of convergence with respect to some specific distance disregarding big jumps of the limit stable process. In the present note we consider the specific situation where all jumps are positive and particles take values in $\R_+. $ We show that in this case it is possible to improve upon the error bounds obtained in \cite{DEE} by using an adhoc distance obtained after applying a concave space transform to the trajectories. The distance we propose here takes into account the big jumps of the limit $ \alpha-$stable subordinator.

\end{abstract}

\end{frontmatter}

\section{Introduction}
We consider a system of interacting particles  $ (X^{N, i}_{t}) ,  t\geq 0, \;  1 \le i \le N , $ taking values in $\R_+^N$  and  evolving according to 
\begin{multline}\label{eq:system}
X^{N, i}_t =  X^{i}_0 +   \int_0^t  b(X^{N, i}_s , \mu_s^N )  ds +  \int_{[0,t]\times\R_+\times\R_+} \psi (X^{N, i}_{s-} )  \indiq_{ \{ z \le  f ( X^{N, i}_{s-}) \}} \pi^i (ds,dz,du) \\
+ \frac{1}{{N}^{ 1/\alpha} }\sum_{ j \neq i } \int_{[0,t]\times\R_+\times\R_+ }u \indiq_{ \{ z \le  f ( X^{N, j}_{s-}) \}} \pi^j (ds,dz,du),
\end{multline} 
where 
$ \mu_t^N = N^{-1}  \sum_{i=1}^N \delta_{X_t^{N, i } } $
is the empirical measure of the system at time $t.$ In the above equation, $(\pi^{i}(ds,dz,du))_{i\geq 1}$ is a family of i.i.d. Poisson measures on $(\R_+\times \R_+\times \R_+)$ having intensity measure $dsdz\nu(du)$ and  $(X_0^{i})_{i\geq 1}$  is an i.i.d. family of $\R_+-$valued random variables independent of the Poisson measures, distributed according to some fixed probability measure $\nu_0 $ on $ (\R_+, {\mathcal B} ( \R_+) ).$ 

The motivation for considering such systems comes originally from neuroscience where each $X^{N, i }_t$ represents the membrane potential of a given neuron belonging to a population of $N$ neurons which are all interconnected. Neurons spike at a rate depending on their membrane potential and when spiking, go back to a resting potential value, while the associated postsynaptic neurons receive the influence of the synaptic weight of the spiking neuron. In particular, interactions appear mostly through spiking, that is, at the jumping times. In formula \eqref{eq:system}, the transition of the spiking neuron's potential to the resting value is represented by the transition $ x^i \mapsto x^i + \psi ( x^i) , $ and the influence of the synaptic weight  by the transition $ x^j \mapsto x^j + u/N^{1/\alpha } ,$ for all $j \neq i , $ where $i $ is the index of the spiking neuron.  

The present note is a companion paper to the recent paper \cite{DEE}, where we have studied the mean field limits of system \eqref{eq:system} in the general case where $ \nu $ is the law of a  random variable belonging to the domain of attraction of a stable law of index $ \alpha \in ( 0, 2 )$ and where the state space of each particle is $ \R.$ 
We have proved there the strong convergence of the system \eqref{eq:system} to the limit system 
\begin{eqnarray}
\label{eq:dynlimintrovrai}
\bar X^i_t &=&  \bar X^i_0+  \int_0^t b(\bar X^i_s, \bar \mu_s)    ds +  \int_{[0,t]\times\R_+ } \psi (\bar X^{ i}_{s-} )  \indiq_{ \{ z \le  f ( \bar X^{ i}_{s-}) \}}  \bar \pi^i (ds,dz ) \nonumber \\
&&\quad \quad + \int_0^t  (\bar \mu_{s-} ( f) )^{1 / \alpha}  d S^\alpha_s  ,\; \; t \geq 0 , i \in \N,
\end{eqnarray}
where 
\begin{equation}\label{eq:projpi}
 \bar \pi^i (ds, dz) = \pi^i ( ds, dz, \R_+) 
\end{equation} 
is the projected Poisson random measure having intensity $ds  dz .$ In the above formula,  $ \bar \mu_s = {\mathcal L} ( \bar X^1_s | S^\alpha_u, u \le s ) $ is the conditional law of any particle given the common noise $ S^\alpha.$ $S^\alpha$ denotes the stable process $ S^\alpha $ having stability index $ \alpha  ;$ this stable process is independent of the collection of Poisson random measures $ (\bar \pi^i (ds, dz))_{ i \geq 1 }$ and of the initial values $ X^i_0, i \geq 1 , $ and it represents a source of common noise for the limit system. As a consequence, the conditional propagation of chaos property holds. 

More precisely we have shown in \cite{DEE} that it is possible to construct a coupling of the finite system $ (X^{N, i })_{ 1 \le i \le N, t \geq 0} $ with the limit system such that for any truncation level $ K > 0, $ for any $ q < \alpha , $ 
\begin{equation}\label{eq:strongerror1}
 \E \left( \indiq_{\{ t < T_K\}} | X_t^{N, i } - \bar X_t^i | \wedge | X_t^{N, i } - \bar X_t^i |^{ q \wedge 1 }  \right) \le C_t (K) r_N , \mbox{ where } T_K = \inf \{ t \geq 0 : | \Delta S_t^\alpha | > K \} , 
\end{equation} 
and where $r_N$ is an explicit rate of convergence.  The constant $ C_t (K) $ depends on the truncation level and blows up as $ K \to \infty.$

The present note restricts attention to the particular subclass of models where $ \alpha < 1 ,$ where all jumps are positive and where $ \nu $ is precisely the law of a strictly stable positive random variable. In this situation we are able to obtain a control of the strong error \eqref{eq:strongerror1} without truncation, that is, by taking formally $ K = \infty . $ We do this by applying a space transform to the particles by means of the function 
\begin{equation}\label{eq:dq}
d_q ( x) = x \wedge x^q , x \geq 0, 
\end{equation}
for some fixed $ q \in (0, \alpha ) ,$ considering the distance $ \E  \left( | d_q (X_t^{N, i }) - d_q(\bar X_t^i) | \right) .$ We heavily employ the fact that $d_q$ is concave on $ \R_+$ and that all jumps are positive, which allows us to deal with the stochastic integral  $ \int_0^t  (\bar \mu_{s-} ( f) )^{1 / \alpha}  d S^\alpha_s$ without truncation -- an approach which is not possible in the general framework considered in \cite{DEE}.  

Our paper follows closely the ideas of \cite{DEE} to which we refer for bibliographic comments and further motivation of the model under study, see also  \cite{andreis_mckeanvlasov_2018}, \cite{cavallazzi}, \cite{DGLP}, \cite{jourdainetal} and \cite{graham92} for papers devoted to propagation of chaos in similar models, and \cite{carmona_mean_2016}, \cite{coghi_propagation_2016}, \cite{dermoune_propagation_2003} and \cite{ELL1} and \cite{ELL2} for articles studying situations with common noise, where the property of conditional propagation of chaos holds. 

With respect to \cite{DEE}, in the present note, we mainly highlight the most important differences, which are mostly due to a different control of the above mentioned stochastic integral, see Lemma \ref{lem:it} below. We quote the main results of \cite{DEE} in order to ease the readability of this paper. Moreover, we give in Lemma \ref{lem:rs} the main argument which shows how to construct the limit stable subordinator $S^\alpha $ explicitly on an extension of the original probability space, using the self similarity property of the stable random variables that are represented by the atoms $u$ of the Poisson random measures $ \pi^j$ of \eqref{eq:system}. 

We now give the precise assumptions on our model parameters and state precisely our main results in this particular subclass of models. 

\subsection{Model assumption, notation and main results}
\subsection*{Notation}
Throughout this paper we shall use the following notation. For $p> 0,$ $\mathcal{P}_p(\R_+)$ denotes the set of probability measures on $\R_+$ that have a finite moment of order~$p$.  
For two probability measures $\nu_1, \nu_2 \in \mathcal{P}_p (\R_+),$ the Wasserstein distance of order $p$ between $\nu_1$ and $\nu_2$  is defined as
$$
W_p(\nu_1,\nu_2)=\inf_{\pi\in\Pi(\nu_1,\nu_2)}\left( \int_{\R_+}\int_{\R_+} |x-y|^p \pi(dx,dy) \right)^{(1/p) \wedge 1 } ,
$$
where $\pi$ varies over the set $\Pi(\nu_1,\nu_2)$ of all probability measures on the product space $\R_+ \times \R_+$ with marginals $\nu_1$ and $\nu_2$. We shall also use the following notation. For any fixed function $ a : \R_+ \to \R_+$ which is injective and any two probability measures $ \nu_1 , \nu_2 $ such that $ a \in L^1 ( \nu_1 ) \cap L^1 ( \nu_2 ) , $ we put 
$$ W_{1,a} ( \nu_1, \nu_2 ) =\inf_{\pi\in\Pi(\nu_1,\nu_2)}\left( \int_{\R_+}\int_{\R_+} |a(x)-a(y)| \pi(dx,dy) \right) .$$
In particular we will also apply this notation for the function $d_q$ introduced in \eqref{eq:dq} above.

Throughout this article, $\nu $ designs the law of a strictly stable and positive  random variable $Y \geq 0 $ of index $ \alpha \in ] 0, 1 [ $ which is characterized by its Laplace transform defined for any $ \lambda \geq 0 $ by 
\begin{equation}\label{eq:1}
 E (e^{- \lambda Y}) = e^{ -  \lambda^\alpha  }.
\end{equation} 
Moreover,  $S^{\alpha}$ denotes the $\alpha$-stable increasing process with $0<\alpha<1.$ It is characterized by the following facts.
\begin{enumerate}
\item All paths of $S^{\alpha}$ are c\`adl\`ag and non-decreasing, and $ S^\alpha_0 = 0.$ 
\item $ S^\alpha $ is a L\'evy process, and $\E e^{-\lambda S_t^{\alpha}}=e^{- t  \lambda^{\alpha}};\; \lambda \geq 0,\; t\geq 0.$
\end{enumerate} 
In particular,  $S^\alpha_1 \sim \nu.$

\subsection*{Model assumptions}

We suppose that 

\begin{ass}\label{ass:b}
The drift function $b : \R_+ \times  {\mathcal P}_1 ( \R_+) \to \R $ satisfies:
\begin{itemize}
\item [ i)]   $ b (0, \mu ) \geq 0 $ for all $ \mu \in  {\mathcal P}_1 ( \R_+).$
\item[ii)] $ \sup_{x\in\R_+,\; \mu \in\cP_1(\R_+)} |b(x,\mu) | < \infty . $ 
\item[iii)] There exists a constant $ C > 0, $ such that for all $ x, \tilde x  \in \R_+  , \mu , \tilde \mu  \in \mathcal{P}_1(\R_+) , $ 
$$ |b(x, \mu ) - b(\tilde x, \tilde \mu) |\le C ( |d_q(x) - d_q(\tilde x) |  +  W_{1, d_q} ( \mu, \tilde \mu )) . $$
\end{itemize}
\end{ass}

Concerning the parameters that are involved in the jumps, we state the following set of assumptions.

\begin{ass}\label{ass:1}
i) $f$ is positive, lowerbounded by some strictly positive constant $ \underline f > 0 $ and bounded. \\
ii) $ \psi$ is bounded and positive.  \\
iii) Moreover, there exists a constant $C > 0, $ such that for all $ x, y \geq 0, $ 
\begin{equation}\label{eq:fq}
 |f(x) - f(y ) | + | \psi ( x) - \psi (y) | \le C  | d_q(x) - d_q(y ) | .
\end{equation} 
\end{ass}

Finally we assume that 
\begin{ass}\label{ass:nu0}
$ \nu_0$ is a probability measure on $ \R_+ $ which  admits a finite first moment, and $ \nu_0 (\{ 0 \} )= 0.$
\end{ass}

To study the limit system, we consider one typical particle $ \bar X_t$ representing the limit system \eqref{eq:dynlimintrovrai}. It  evolves according to
\begin{equation}\label{eq:limitequation}
\bar X_t =  X_0 +   \int_0^t  b(\bar X_s , \bar \mu_s )  ds +  \int_{[0,t]\times\R_+ } \psi (\bar X_{s-} )  \indiq_{ \{ z \le  f ( \bar X_{s-}) \}} \bar \pi (ds,dz) 
+ \int_{[0,t] } \left( \bar \mu_{s-}(f)  \right)^{1/ \alpha}   d S_s^\alpha , 
\end{equation} 
where $X_0 \sim \nu_0, $ $ \bar \mu_s = {\mathcal L} ( \bar X_s | S^\alpha_u, u \le s ) ,$ $ \bar \pi $ is a Poisson random measure on $ \R_+ \times \R_+ $ having intensity $ ds dz,$  and where $ S^\alpha , \bar \pi  $  and $ X_0$ are independent.

We now collect some known results that have already been established in \cite{DEE}.  
\begin{prop}\label{prop:wellposedfinite}
Under Assumptions \ref{ass:b}, \ref{ass:1} and \ref{ass:nu0}, equations \eqref{eq:system} and  \eqref{eq:limitequation} possess a unique strong solution taking values in $\R_+^N  $ (in $\R_+,$ respectively) satisfying that for all $ T > 0 $ and for any $ 1 \le i \le N, $   
\begin{equation}\label{eq:apriorifinite}
\E  \left(\sup_{ t \le T } |d_q( X_t^{N, i })| \right) +\E \left(  \sup_{ s \le T} | d_q( \bar X_s)| \right) \le C_T  
\end{equation}
for a constant $C_T$ that does not depend on $N$ neither on $i.$ 
\end{prop}
The above statement follows from Theorem 2.12 and Theorem 2.13 of \cite{DEE}.

\begin{rem}\label{rem:33}
Since $X_0 , X^i_0 \geq 0 , $ almost surely, $ X_t^{N, i } , \bar X_t  \geq 0 $ for all $ 1 \le i \le N,$ $t \geq 0.$ Indeed, all jumps of the process are positive such that the process has to hit $ ] - \infty , 0] $ in a continuous way, that is, due to the presence of the drift term $b; $ which is not possible since $ b( 0, \mu ) \geq 0 $ for all $ \mu,$ by assumption.
\end{rem}
 
Before stating the main theorem of the present note, we slightly strengthen Assumption \ref{ass:nu0} to
\begin{ass}\label{ass:nu0bis}
$ \nu_0$ is a probability measure on $ \R_+ $ which  admits a finite moment of order $ (2 \alpha) \vee 1  ,$ and $ \nu_0 (\{ 0 \} )= 0.$
\end{ass}

\begin{thm}\label{strongconvergence}
Grant Assumptions  \ref {ass:b}, \ref{ass:1} and \ref{ass:nu0bis} and fix $N\in\N^*.$   Then it is possible to construct, on an extension of $ (\Omega, {\mathcal A}, {\mathbf P}), $  a one-dimensional stable process $(S_t^{N, \alpha})_{t \geq 0},$ depending on $ N,$ which is independent of  the initial positions $X^{i}_0, i \geq 1,$ and of $ (\bar \pi^i)_{i \geq 1} , i \geq 1 , $
such that  the following holds. Denoting $ (\bar X^{N, i })_{ i \geq 1 } $ the strong solution of the limit system \eqref{eq:limitequation}, driven by $S^{N, \alpha}$ and by $ (\bar \pi^i )_{ i \geq 1},$ we have for every $T>0,$ and for any $ 1 \le i \le N, $ for all $ t \le T,$ 
\begin{multline}
\label{strongbound}
 \E (  | d_q(X^{N,i}_t) -d_q(\bar X^{N, i}_t)  | ) \leq \\
 C_T  \left( N^{ 2 ( 1 - q/\alpha) -  \frac{ 4 (1 - \frac{q}{\alpha}) + q }{ 2 ( 1 - \frac{q}{\alpha} ) + q + 2  } }+ N^{- \frac{q}{\alpha} \frac{ 4 (1 - \frac{q}{\alpha}) + q }{ 2 ( 1 - \frac{q}{\alpha} ) + q + 2  }} +  [ N^{ - q}  \indiq_{\{ q < \frac12\}} + N^{- 1/2 }\indiq_{\{ q > \frac12\}}] \right) .
 \end{multline}
 \end{thm}

\begin{rem}
\eqref{strongbound} holds for any $ q < \alpha.$ To understand what is the leading order, let us take $ q = \alpha $ such that the right hand side of \eqref{strongbound} is given by 
$$ N^{ - \frac{\alpha}{2 + \alpha}} +  [ N^{ - \alpha }  \indiq_{\{ \alpha < \frac12\}} + N^{- 1/2 }\indiq_{\{ \alpha > \frac12\}}] .$$
The leading order of convergence is therefore given by 
$$ C_T N^{ - \frac{\alpha}{2 + \alpha}}.  $$ 

Of course, this rate of convergence is only the limit rate, obtained for $q = \alpha$ which is not a valid choice for $q.$ Therefore, the actual rate of convergence that we obtain is any rate that is slightly slower than the above one.
\end{rem}

\begin{rem}
It is possible to extend the result of Theorem \ref{strongconvergence} to the more general situation where $ \nu $ is not precisely the law of a strictly stable random variable but does only belong to the domain of attraction of it. 
We have decided to not include this case in the present note; it is treated in the general framework (where jumps are not obliged to be positive) in \cite{DEE}. 
\end{rem}

\subsection{Discussion of our assumptions}
\begin{rem}
The function $d_q$ is concave on $\R_+ $ and satisfies $ d_q(0)= 0.$ In particular, this implies that it is sub-additive such that 
$$ | d_q(y) - d_q(y') | \le d_q ( |y - y'|) \le  |y - y'|^q $$ 
for all $ y, y' \geq 0 .$ Therefore,  item iii) of Assumption \ref{ass:b} implies that 
$$ |b(x, \mu ) - b(\tilde x, \tilde \mu) |\le C ( |d_q(x) - d_q(\tilde x) | + W_q( \mu , \tilde \mu)).$$
Thus, for any fixed $x,$ the mapping  $ \mu \mapsto b( x, \mu ) $ is Lipschitz continuous with respect to the Wasserstein $q-$distance - but the latter continuity is not quite sufficient for our purposes since we have to work with the modified Wasserstein $d_q-$distance which is (slightly) stronger.
\end{rem}

\begin{rem}
We will often use that Assumption \ref{ass:1} implies that 
\begin{equation}\label{eq:fq2}
 |f(x) - f(y ) | + |\psi ( x) - \psi (y ) |  \le C |x-y|^q ,
\end{equation} 
for any $ x, y \geq 0.$ 
We will also  repeatedly use  that for any two probability measures $ \mu , \nu  \in \mathcal{P}_1(\R_+) ,$ 
\begin{equation}\label{eq:controlmuf}
|( \mu ( f))^{1 / \alpha } - ( \nu ( f) )^{1/\alpha } |    \le C  | \mu ( f) - \nu ( f) | ,
\end{equation}
where we used that $ z \mapsto z^{ 1/ \alpha } $ is Lipschitz on $ [ 0, \|f\|_\infty ]  .$ Finally, observe that we have the upper bound 
\begin{equation}\label{eq:controlmuf2}
|( \mu ( f))^{1 / \alpha } - ( \nu ( f) )^{1/\alpha } | \le C W_q ( \mu, \nu ) ,
\end{equation}
which follows from \eqref{eq:controlmuf} and \eqref{eq:fq2} since
$$  | \mu ( f) - \nu ( f) | \le \int | f(x) - f(y ) | \pi (dx, dy )  \le C \int |x-y |^q \pi ( dx,dy ) ,$$
which holds for any coupling $ \pi (dx, dy ) $ of $ (\mu, \nu ) , $ such that taking the optimal coupling for $ W_q,$ the assertion follows.

\end{rem}

To check that the drift function  $b$ actually satisfies item iii) of of Assumption \ref{ass:b}, it is convenient to  rely on the notion of functional derivative, introduced for instance in \cite{Jourdain}.
Remember that a function $B:{\cP}_1(\R_+)\to\R_+; $ $\nu\mapsto B(\nu)$ admits a functional derivative $\delta_{\nu} B: \R_+\times {\cP}_1(\R_+)\to\R, $ $ ( y, \nu ) \mapsto \delta_{\nu} B ( y , \nu ),$ if for each $ \nu \in {\cP}_1 ( \R_+), $ $ \sup_y \left| \delta_{\nu} B ( y , \nu ) \right| / ( 1 + |y|) < \infty $ and 
$$B(\nu)-B(\mu)=\int_0^1\int_{\R_+} \delta_{\nu} B(y,(1-\lambda)\mu+\lambda \nu)(\nu-\mu)(dy)d\lambda.$$

\begin{rem}\label{cor:couplingboundb}
Let us suppose that $ x \mapsto b (x, \mu ) $ is $d_q-$Lipschitz, uniformly in $ \mu ,$ and that for all $x\in\R_+,$ $\mu \mapsto b (x,\mu)$ admits a functional derivative $\delta_{\mu} b(x,y,\mu)$ which satisfies
\begin{equation}\label{eq:boundb}
 \sup_{x\in\R_+;\; \mu\in\cP_1(\R_+)}|\delta_{\mu} b(x,y,\mu)-\delta_{\mu} b(x,y',\mu)|\leq C |d_q(y)-d_q(y')|,
\end{equation} 
for all $ y \in \R_+$ Then item iii) of Assumption \ref{ass:b} holds true. 

If $ \delta_\mu b $ does not depend on the measure variable, then $b$ depends linearly on $ \mu $ and can be written as 
\begin{equation}\label{eq:linearb}
b (x, \mu ) = \int_{\R_+} \tilde b ( x, y ) \mu (dy), \mbox{ where } \tilde b ( x, y ) = \delta_\mu b (x, y, \mu).
\end{equation}
\end{rem}

The proof of the fact that \eqref{eq:boundb} implies item iii) of Assumption \ref{ass:b} is postponed to the Appendix.

\begin{cor}
Under the assumptions of Theorem \ref{strongconvergence}, if $b$ is linear in $ \mu, $ that is, \eqref{eq:linearb} holds and if moreover $ \tilde b  $ is bounded, then we obtain the better rate 
\begin{equation}
\label{strongboundbetter}
 \E (  | d_q(X^{N,1}_t) -d_q(\bar X^{N, 1}_t)  | ) \leq C_T  \left(  N^{ 2 ( 1 - q/\alpha) -  \frac{ 4 (1 - \frac{q}{\alpha}) + q }{ 2 ( 1 - \frac{q}{\alpha} ) + q + 2  } }+ N^{- \frac{q}{\alpha} \frac{ 4 (1 - \frac{q}{\alpha}) + q }{ 2 ( 1 - \frac{q}{\alpha} ) + q + 2  }} + N^{ - 1/2}  \right) .
 \end{equation} 
\end{cor}

\section{Main steps of the proof of Theorem \ref{strongconvergence}}

\subsection{A smooth version of $d_q$}
To compare $d_q ( X^{N, i  }_t ) $ to $ d_q ( \bar X^{N, i }_t), $ it is convenient to replace the function $d_q$ which has a singularity in $ x= 1 $ by a smooth version $a$ that we introduce now. This will allow us to apply Ito's formula in order to calculate $d_q(X^{N,1}_t) -d_q(\bar X^{N, 1}_t).$ Somewhere below we will also consider approximating processes which might take negative values. This is why we define the function $a$ on the whole real line $ \R$ and not only on $\R_+ .$

\begin{prop}\label{prop:a}
There exists a function $a : \R \to \R$ belonging to ${\cC}^1(\R)$ which satisfies
\begin{itemize}
\item[i)] $a$ is concave on $ \R_+$ and $a(0)=0.$
\item [ii)]  There exists a constant $C > 0 $ such that for all $ x, y  \in \R  , $  $|a' (x) - a'(y) | \le  C |a(x) - a(y) |. $ 
\item[iii)] $a$ is Lipschitz.
\item[iv)] $a$ is strictly increasing.
\item [v)] There exists a constant $ c$ (positive or negative) such that $a(x)=c+ x^q$ for all $x\geq 1.$
\item[vi)] For some $0 <  c_1 < c_2$ we have  $ c_1 | d_q ( x) - d_q(y) | \le |a(x) - a(y) | \le c_2 |d_q ( x) - d_q (y) | $ for all $ x, y \geq 0.$ 
\end{itemize}

\end{prop}

\begin{proof}
There are several ways of constructing such a function $a.$ Here is one possibility. We put 
\begin{equation}\label{eq:a} a( x) = \left\{ 
\begin{array}{ll}
 ( q + q (1-q)) x - \frac12 q (1-q) x^2 & 0 \le x \le 1 \\
 x^q  + (q+ \frac12 q (1-q) -1) & x \geq 1 
 \end{array}
 \right\} ,
 \end{equation}
and finally $ a (-x) = - a ( x) $ for all $ x \le 0.$ Then it is easy to check that this function is actually $C^2 $ and that it satisfies  i) and all points iii)--vi). Concerning point ii), notice that by construction, $ a'' / a' $ is bounded, which 
implies ii).  
\end{proof}

Let us now discuss some obvious properties of the function $a.$

\begin{rem}
It follows from item $i)$ of Proposition \ref{prop:a}  that for all $ x, y \geq 0,$
$$ a(x+y ) - a(x) \le a(y) - a(0)   =  a(y),$$
since $ a (0 ) = 0,$ which means that for all $ x, y \geq 0,$ 
\begin{equation} a(x+y) \le a(x) + a(y) ,
\end{equation}
that is, the function $a$ is sublinear on $ \R_+ .$ 

We shall also use that for all $ x, y \geq 0  , $ and all $ u \geq 0, $
\begin{equation}
| a ( x + u ) - a(y + u )| \le |a(x) - a(y ) | ,
\end{equation} 
which follows from the fact that $a$ is concave on $ \R_+ . $ Let us notice furthermore that $a'$ is bounded, which is a consequence of $iii).$ 
\end{rem}
In the sequel the following property of $a$ will also be used. 

\begin{lem}\label{lem:useful}
Let $ x \geq 0 $ and $ y  \in \R.$ Then 
\begin{equation}\label{eq:goodtoknow}
(i)\;  |a(x) - a(y) | \le 2 a( |x-y|);\quad\quad  \mbox{ and}\quad\quad  (ii)\,  \forall z \geq 0,\;   |a( x+z ) - a(y+z)|\le 2 a(2 |x-y|) .
\end{equation} 
\end{lem}
The proof of these facts is postponed to the Appendix. 

Item $vi)$ of Proposition \ref{prop:a} allows us to control  $| d_q(X^{N,1}_t) -d_q(\bar X^{N, 1}_t)|$ in terms of $ |a(X^{N,1}_t) -a(\bar X^{N, 1}_t) |,$ and this is what we are going to do in the sequel. We now present the main steps of the proof following the ideas of \cite{DEE}.   

\subsection{Rewriting the interaction  term of the finite particle system in terms of subordinators}\label{sec:coupling}
To prove Theorem \ref{strongconvergence}, we follow the same strategy as the one already used in \cite{DEE} and cut time into time slots of length $ \delta > 0.$ We will choose $ \delta = \delta (N )$ such that, as $N \to \infty, $ $\delta (N) \to 0$ and (at least) $ N \delta (N) \to \infty$ such that per time interval, the total jumping activity still tends to infinity. The precise choice will be given in \eqref{eq:delta} below. 

We write $\tau_s=k\delta$ for $k\delta < s \leq(k+1)\delta,\,  k\in\N, s > 0.$ The main step of the proof is to replace  the interaction term 
\begin{equation}\label{eq:atn}
A_{t}^{N}:=\frac 1{N^{1/\alpha}}\sum_{j=  1}^N\int_{[0,t]\times\R_+\times\R_+ }u \indiq_{ \{ z \le  f ( X^{N, j}_{s-}) \}} \pi^j (ds,dz,du)
\end{equation}
by its time-frozen version 
\begin{equation}\label{eq:atndiscreti}
A_{t}^{N, \delta}:=\frac 1{N^{1/\alpha}}\sum_{j=  1}^N\int_{[0,t]\times\R_+\times\R_+ }u \indiq_{ \{ z \le  f ( X^{N, j}_{\tau_s}) \}} \pi^j (ds,dz,du) .
\end{equation}
Freezing time in the jumping rate, that is, working with a jumping rate $f ( X^{N, i}_{\tau_t})$ which is constant on each time interval $ (k \delta, (k+1 ) \delta ] $ enables us to separate the randomness present in the jump heights (the variables $u$ in \eqref{eq:system}) from the randomness present in the acceptance-rejection scheme represented by $ \bar \pi^i .$ 

In what follows we give the main argument that leads to the construction of the limit stable subordinator $S^\alpha .$ 
We denote,  for any $ k \geq 0, $ 
$$P^{N, \delta}_{k , k+1}:= \sum_{j= 1}^N\int_{]k \delta,(k+1) \delta]\times \R_+}\indiq_{\{z\leq f(X^{N, j }_{k \delta }) \}} \bar \pi^j(dv,dz)$$
and
$$A^{N,\delta}_{k, k+1} := \sum_{j= 1 }^N\int_{]k \delta,(k+1) \delta]\times \R_+\times \R_+}\frac u{N^{1/\alpha}}\indiq_{\{z\leq f(X^{N, j }_{ k \delta })\}} \pi^j(dv,dz,du),$$
which is the increment of the time frozen process $A^{N, \delta} $ over one time interval $ (k \delta, (k+1 ) \delta ] .$ Then we have the following key result.

\begin{lem}\label{lem:rs}
For any fixed $ k \geq 0,$ there exists a random variable $ S^{\alpha , \delta}_{ k, k+1} $ defined on an extension of $( \Omega, {\mathcal A}, {\mathbf P} ),$ satisfying the following properties.\\
i) $ S^{\alpha, \delta }_{ k, k+1} \sim \nu.$\\
ii) $ S^{\alpha , \delta}_{ k, k+1}$ is
 independent of ${\mathcal F}_{k \delta},$ of $P^{N,\delta}_{k, k+1}$ and of $ (\bar \pi^i)_{i \geq 1 }.$\\
iii) We have the representation 
\begin{equation}\label{eq:representationbrute}
A^{N,\delta}_{k, k+1}=\left( \frac{P^{N,\delta}_{k, k+1}}{N}\right)^{ 1 / \alpha}  S^{\alpha , \delta}_{ k, k+1} .
\end{equation}
\end{lem}

\begin{proof}
Let 
\begin{equation}\label{eq:Yk}
  S^{\alpha, \delta}_{ k, k+1} :=(\frac N{P^{N,\delta}_{k, k+1}})^{1/\alpha} A^{N,\delta}_{k,k+1} \indiq_{\{P^{N, \delta}_{k , k+1}\neq 0\}}+ \bar Y \indiq_{\{P^{N, \delta}_{k , k+1}= 0\}},
  \end{equation}
where $\bar Y \sim \nu,$ independent of ${\mathcal F}_\infty .$
In what follows we show that  $  S^{\alpha, \delta}_{ k, k+1}$ follows indeed the law $\nu$ and that it is independent of $\F_{k \delta}, $ $ P^{N, \delta}_{k , k+1}$ and $ (\bar \pi^i)_{i \geq 1 }$.

To do so we propose a representation of $A^{N, \delta}_{k , k+1}$ in terms of the atoms of the Poisson random measures $ \pi^j.$ Using basic properties of Poisson random measures and the fact that $f$ is bounded, we only need to consider  $\pi^j_{| \R_+ \times [0, \| f\|_\infty ] \times \R_+}$ which can be represented as
$$\pi^j(]k \delta,(k+1) \delta ]\times A\times B)=\sum_{N_{k \delta}^j<n\leq N_{(k+1) \delta}^j}\indiq_A(Z_n^j)\indiq_B(U_n^j),$$ 
for any$ A , B \in {\mathcal B} (R_+), A \subset [0, \|f\|_\infty ].$ 
Here, $(N_s^j)_{s\geq 0}$ is a Poisson process with intensity $\|f\|_{\infty},$ and $(Z_n^j)_{n\geq 0},$ $(U_n^j)_{n\geq 0}$ are two independent i.i.d. sequences, both independent  of $(N_{s}^j)_s$, such that  $Z_n^j$ are  uniformly distributed on $[0, \|f\|_{\infty}]$ and $U_n^j\sim\nu$ on $\R_+.$   For different $j\geq 1,$ $(N_{s}^j)_s, (Z_n^j)_n, (U_n^j)_n$ are independent. 
With this definition  we can represent $A^{N, \delta}_{k \delta,(k+1) \delta} $ and $ P^{N, \delta}_{k \delta,(k+1) \delta} $ as
$$P^{N, \delta}_{k ,k+1 }=   \sum_{j=1}^N \sum_{N_{k \delta}^j<n\leq N_{(k+1)\delta}^j} \indiq_{\{Z_n^j\leq f(X^{N,j}_{k \delta})\}}\quad{a.s.} $$
and
$$ A^{N, \delta}_{k ,k+1} =  \frac{1}{{N}^{ 1/\alpha} }\sum_{ j =1}^N \sum_{N_{k \delta}^j<n\leq N_{(k+1) \delta}^j}U^j_n \indiq_{\{Z_n^j\leq f(X^{N,j}_{k \delta})\}}\quad a.s.. $$
Note that $P^{N, \delta}_{k ,k+1 }$ is the total number of non-zero terms in the sum defining $A^{N, \delta}_{k ,k+1} .$  
The fact that $  S^{\alpha, \delta}_{ k , k+1}\sim\nu$  follows from 
\begin{multline*}
\E \left (e^{-\lambda   S^{\alpha , \delta}_{ k , k+1  }} |\; \F_{k \delta}; N^j_{(k +1) \delta};\;  Z^j_n;\;  \ N_{k \delta }^j<n\leq N_{(k+1) \delta}^j, \; j=1,\ldots N\right)\indiq_{\{P^{N, \delta}_{k ,k+1}\neq 0\}}\\
=e^{-\lambda^{\alpha}}\indiq_{\{P^{N, \delta}_{k ,k+1}\neq 0\}}.
\end{multline*}
The above statement also shows that $ S^{\alpha , \delta}_{ k, k+1 }$ is independent of $\F_{k \delta}$ and $P^{N, \delta}_{k ,k+1}$.
Moreover, since for any $t\geq 0$ and $A\in\B(\R_+),$ $ A \subset [0, \|f\|_\infty ], $ 
 $$\bar \pi^j ([0,t]\times A) =\sum_{n\leq N_t^j}\indiq_A(Z_n^j),  $$
the same conditioning implies that $ S^{\alpha, \delta}_{  k , k+1}$ is independent of $\bar \pi^j .$ This concludes the proof.
\end{proof}

The above lemma shows how to use, on each time interval $ (k \delta, (k+1) \delta ], $  the jump heights $u$ to construct an increment $ S^{\alpha, \delta}_{ k , k+1}$ of the limit stable subordinator. Concatenating these increments, it is possible to construct a subordinator defined on the whole positive axis (this is the content of Proposition 4.3 of \cite{DEE}).  Using moreover the strong law of large numbers to replace the random and conditionally Poisson distributed factor $ \frac{P^{N,\delta}_{k, k+1}}{N}$ in \eqref{eq:representationbrute} by its intensity $ \sum_{j=1}^N f( X^{N, j}_{ k \delta }) \delta ,$ we have proved in \cite{DEE} the following theorem. 

\begin{thm}\label{theo:first}{(Theorem 4.1 of \cite{DEE})}
\begin{description}
\item {i)}
For each $N$ and $ \delta < 1,$ there exists a stable subordinator $ S^{N, \alpha}, $ defined on an extension of $ (\Omega, {\mathcal A}, {\mathbf P}) ,$  independent of $ \bar \pi^i , i \geq 1, $ and of $ X^{N, i }_0, i \geq 1, $ 
such that for any $N\in\N,$ 
\begin{equation*}
A_t^N = \int_0^t (\mu^N_{ \tau_s} (f) )^{ 1 / \alpha }  d S^{N, \alpha}_s + R^{N}_t,
\end{equation*} 
where the error term satisfies for all $ t \le T,$ 
\begin{equation*}
 \E ( 
 d_q ( | R_t^{N } |)  ) \le C_T \left( N^{ 2(1 - q/\alpha)} \delta+ (N \delta)^{ - q/2} \delta^{ q/\alpha - 1 }  +\delta^{q/\alpha} 
   \right) .
\end{equation*}

\item {ii)}
With $ S^{N, \alpha}  $ the stable subordinator of item i), we have for any $ 1 \le i \le N $ and any $ t \geq 0, $ 
\begin{multline*}
X^{N, i}_t =  X^{N,i}_0 +   \int_0^t  b(X^{N, i}_s , \mu_s^N)  ds +  \int_{[0,t]\times\R_+} \psi(X^{N, i}_{s-})   \indiq_{ \{ z \le  f ( X^{N, i}_{s-}) \}} \bar \pi^i (ds,dz) \\
+\int_0^t (\mu^N_{\tau_s} (f) )^{ 1 / \alpha }  d S^{N, \alpha}_s + \tilde R^{N, i }_t,
\end{multline*} 
where the error term satisfies for all $ t \le T,$ 
\begin{equation*}
 \E ( 
d_q( | \tilde  R_t^{N , i } |)  ) \le C_T \left( N^{ 2(1 - q/\alpha)} \delta+ (N \delta)^{ - q/2} \delta^{ q/\alpha - 1 }  +\delta^{q/\alpha}  
  + N^{ - q/\alpha}  \right) .
\end{equation*}
\end{description}
\end{thm}
\begin{proof}[Proof of Theorem \ref{theo:first}]
Item i) of the above theorem is a consequence of Theorem 4.1 in \cite{DEE}.   Concerning item ii), observe that
\begin{multline*}
X^{N, i}_t =  X^{N,i}_0 +   \int_0^t  b(X^{N, i}_s , \mu_s^N)  ds +  \int_{[0,t]\times\R_+} \psi(X^{N, i}_{s-})   \indiq_{ \{ z \le  f ( X^{N, i}_{s-}) \}} \bar \pi^i (ds,dz) \\
+ A_t^N - \frac{1}{N^{ 1/\alpha} } \int_{ [0, t ]  \times \R_+ \times \R_+ } u \indiq_{ \{ z \le  f ( X^{N, i}_{s-}) \}}  \pi^i (ds,dz, du ) ,
\end{multline*} 
such that, by sub-additivity of the function $d_q  $ and since $\int_{ [0, t ]  \times \R_+ \times \R_+ } u \indiq_{ \{ z \le  f ( X^{N, i}_{s-}) \}}  \pi^i (ds,dz, du )$ contains only an almost surely finite number of terms, the assertion follows from the fact that 
$$ d_q( | \tilde  R_t^{N , i } |)  \le  d_q( |   R_t^{N  } |) + N^{ - q/\alpha} \left|  \int_{ [0, t ]  \times \R_+ \times \R_+ } u \indiq_{ \{ z \le  f ( X^{N, i}_{s-}) \}}  \pi^i (ds,dz, du ) \right|^q  $$
together with the control 
$$  \left|  \int_{ [0, t ]  \times \R_+ \times \R_+ } u \indiq_{ \{ z \le  f ( X^{N, i}_{s-}) \}}  \pi^i (ds,dz, du ) \right|^q  
 \le  \int_{ [0, t ]  \times \R_+ \times \R_+ } |u|^q  \indiq_{ \{ z \le  f ( X^{N, i}_{s-}) \}}  \pi^i (ds,dz, du ) .
 $$
Taking expectation implies the assertion, since $ |u|^q \in L^1 ( \nu ) $ and since $ f$ is bounded. 
\end{proof} 
Notice that the representation of item ii) above is almost the one we are looking for, except that the integrand in the stochastic integral term is given  in its time-discretized version; that is, we have the term $ \int_0^t (\mu^N_{\tau_s} (f) )^{ 1 / \alpha }  d S^{N, \alpha}_s$ 
appearing instead of $\int_0^t (\mu^N_{ s-} (f) )^{ 1 / \alpha }  d S^{N, \alpha}_s.$ Moreover, we need to control the distance of the finite system to its associated limit by means of the function $d_q(\cdot) $ (or rather by means of the function $a$ introduced in \eqref{eq:a} and related to $d_q$ according to Proposition \ref{prop:a}). The following result therefore extends the representation of $ X_t^{N, i }, $ obtained in item ii) of Theorem \ref{theo:first}, to a representation of $ a ( X_t^{N, i }).$ 

\begin{thm}\label{theo:2main}
Grant Assumptions~\ref{ass:b},~\ref{ass:1} and \ref{ass:nu0bis}.  Let $M(ds, dx) $ be the jump measure of the stable subordinator $ S^{N, \alpha}$ of Theorem \ref{theo:first}. For any $ 1 \le i \le N, $ 
\begin{multline}\label{eq:axfinite}
a(X^{N, i}_t) =  a(X^{i}_0) +   \int_0^t  a' ( X^{N, i}_s) b(X^{N, i}_s , \mu_s^N)  ds \\
+  \int_{[0,t]\times\R_+}[  a(X^{N, i}_{s-} +  \psi(X^{N, i}_{s-})) - a(X^{N, i}_{s-}) ]   \indiq_{ \{ z \le  f ( X^{N, i}_{s-}) \}} \bar \pi^i (ds,dz) \\
+\int_{[0, t ]  \times \R_+} [a(X^{N, i}_{s-} + (\mu^N_{ s-} (f) )^{ 1 / \alpha }  x ) - a(X^{N, i}_{s-}) ]   M  ( ds, dx) + \bar R_t^{N,i} ,
\end{multline}
where 
$$ \E | \bar R_t^{N, i }|  \le C_t  \left(N^{ 2(1 - q/\alpha)} \delta + (N \delta)^{ - q/2} \delta^{ q/\alpha - 1 }  +\delta^{q/\alpha}   + N^{ - q/\alpha}  \right)  .$$
\end{thm}

The proof of the above theorem is given in Subsection \ref{sec:24} below. 

\begin{rem}
Before proceeding with the proof, let us make the following observation. Of course we could directly apply It\^o's formula to $ a(X^{N, i }_t). $ The interaction term would then read as
$$ \sum_{j \neq i } \int_{\R_+ \times \R_+ \times \R_+ }[ a ( X^{N, i }_{t-} + \frac{u}{N^{1/\alpha}}) - a ( X^{N, i }_{t-}) ] \indiq_{ \{ z \le  f ( X^{N, j}_{t-}) \}} \pi^j ( dt, dz, du ).$$  
It is however not so evident how to treat this term. Applying Taylor's formula gives 
$$  a ( X^{N, i }_{t-} + \frac{u}{N^{1/\alpha}}) - a ( X^{N, i }_{t-})  = a'   ( X^{N, i }_{t-} ) \frac{u}{N^{1/\alpha}} + \mbox{ some remainder terms of order } \frac{u^2}{N^{2 /\alpha } }.$$ 
However, $ u $ follows a stable law and does not possess any moment such that a control of such remainder terms does not seem to be feasible. It is for this reason that we propose a ``global'' approach in the proof given below that relies on the representation of the cumulated influence of the interaction terms obtained in item i) of Theorem \ref{theo:first} above. 
\end{rem}

\subsection{A useful technical result for the proof of Theorem \ref{theo:2main}}\label{eq:sectechique}
The proof of Theorem \ref{theo:2main} relies on a technical lemma that we 
state in this subsection for generic processes $X $ and $ \tilde X $ with associated measures $ \mu $ and $ \tilde \mu $ which can be either the associated empirical distributions or the respective conditional laws given $ S^{N, \alpha }.$ 
In the sequel, $ M $ will always denote the Poisson random measure associated to the jumps of $S^{N, \alpha}  .$ We suppose that all assumptions imposed in Theorem \ref{strongconvergence} are satisfied. 

 \begin{lem}\label{lem:it}
Introduce the process $ I_t (X, \tilde X, \mu , \tilde \mu) = I_t $ given by 
$$
I_t := \int_{[0, t ] \times \R_+  } [  a ( \tilde X_{s-} + \tilde \mu^{1/ \alpha}_{s-}(f) z) - a ( \tilde X_{s-} )] -  [  a ( X_{s-} + \mu^{1/ \alpha}_{s-}(f)z ) - a ( X_{s-} )]   M ( ds, dz) .
$$ 
If both processes $ X$ and $ \tilde X$ take non-negative values, then we have
\begin{equation}\label{eq:samex2}
 \E  |I_t (X, \tilde X, \mu , \tilde \mu) |  \le C \E \int_0^t \left( \left| a(\tilde X_s  ) - a(X_s )\right| +  | \tilde \mu_{s}(f)    -  \mu_{s}(f) |\right) ds .
\end{equation} 
If only one of the two processes takes non-negative values, then 
\begin{equation}\label{eq:samex2bis}
 \E  |I_t (X, \tilde X, \mu , \tilde \mu) |  \le C \E \int_0^t \left( \left| a(2 | \tilde X_s -X_s| )\right| +  | \tilde \mu_{s}(f)    -  \mu_{s}(f) |\right) ds .
\end{equation} 
\end{lem}

\begin{rem}
The fact that this lemma holds is the main difference with respect to \cite{DEE} where we had to truncate big jumps of the driving L\'evy process. 
\end{rem}

\begin{proof} 
We first give the proof in case $ X_t , \tilde X_t \geq 0 $ for all $ t. $ Recall that we have that $a(y)=c+ y^q$ if $y \geq 1.$ Moreover, $f$ is lowerbounded by $ \underline f$ such that the arguments $ \tilde \mu_{s-}(f)^{1/\alpha}z $ and $  \mu_{s-}(f)^{1/\alpha}z $ are both lower bounded by $1$ if $ z \geq  \underline f^{- 1/\alpha } . $ We therefore cut $I_t $ into two parts, a first one, $ I_t^{1}, $ corresponding to small jumps $ z \in ] 0, \underline f^{- 1/\alpha } [,$ and a second one, $ I_t^{ 2}, $ corresponding to big jumps $ z \geq   \underline f^{- 1/\alpha }.$ 
We then rewrite 
\begin{multline}\label{eq:scenario1cite}
        \E[ | I^1_t|]  \le \\
         \E\left[ \int_{[0,t]\times ]0, \underline f^{- 1/\alpha } [ } \left| a(\tilde X_{s} + \tilde \mu^{1/\alpha}_{s}(f) z) - a(\tilde X_{s})\right.\right.
          \left.\left. - \left( a(X_{s} + \mu^{1/\alpha}_{s}(f) z) - a(X_{s})\right)\right|  \nu^\alpha(dz)ds \right]\\
         = \E\left[  \int_{[0,t]\times ]0, \underline f^{- 1/\alpha } [  }  \Big|\int_0^1 \left( a'(\tilde X_s + u \tilde \mu^{1/\alpha}_{s}(f) z ) \tilde \mu^{1/\alpha}_{s}(f)  \right.\right. \\
         \quad \quad  \left.  \left.- 
        a'(X_s + u (\mu^{1/\alpha}_{s}(f) z ) \mu^{1/\alpha}_{s}(f) \right) du
        \Big| z \nu^\alpha(dz) ds\right] .
\end{multline}
Using the boundedness of $ f $ and of $ a', $  this can be upper bounded by 
\begin{multline}\label{eq:ashiftprime} 
C\E\Biggl[ \int_{[0,t]\times ]0, \underline f^{- 1/\alpha } [  \times [0,1]}   \biggl[ \left| a'(\tilde X_s + u \tilde \mu^{1/\alpha}_{s}(f) z )  -  a'( \tilde X_s + u \mu^{1/\alpha}_{s}(f) z ) \right| \Bigr. \Bigr.  \\
\Bigl. \Bigl. + \left| a'(\tilde X_s + u \mu^{1/\alpha}_{s}(f) z ) - a'(X_s + u \mu^{1/\alpha}_{s}(f) z )\right|  \Bigr. \Bigr.  \\
\Bigl. \Bigl.+ | \tilde \mu^{1/\alpha}_{s}(f) - \mu^{1/\alpha}_{s}(f) | \biggr]  du z  \nu^\alpha(dz)ds \Biggr] .
\end{multline}
Here, in the first term we use that $ a'$ is Lipschitz and that $ uz  \le C $ on $ \{ z \le   \underline f^{- 1/\alpha }   \} $ to upper bound 
\[
 \left| a'(\tilde X_s + u \tilde \mu^{1/\alpha}_{s}(f) z )  -  a'( \tilde X_s + u \mu^{1/\alpha}_{s}(f) z ) \right| \le C | \tilde \mu^{1/\alpha}_{s}(f)  -  \mu^{1/\alpha}_{s}(f)|  
 \le C   | \tilde \mu_{s}(f))    -  \mu_{s}(f) | ,
\]
where we have also used \eqref{eq:controlmuf}. 

To deal with the second term in \eqref{eq:ashiftprime}, recall that $ |a' ( x) - a'(y ) | \le C |a(x) - a(y ) |,   $ such that  
$$ \left| a'(\tilde X_s + u \mu^{1/\alpha}_{s}(f) z ) - a'(X_s + u \mu^{1/\alpha}_{s}(f) z )\right|  \le C   \left| a(\tilde X_s + u \mu^{1/\alpha}_{s}(f) z ) - a(X_s + u \mu^{1/\alpha}_{s}(f) z )\right|.$$
Now, using that $a$ is concave and   $ u \mu^{1/\alpha}_{s}(f) z > 0, $ we have 
$$  \left| a(\tilde X_s + u \mu^{1/\alpha}_{s}(f) z ) - a(X_s + u \mu^{1/\alpha}_{s}(f) z )\right| \le \left| a(\tilde X_s  ) - a(X_s )\right|.$$ 
Since $ z $ is integrable on $ ]0, \underline f^{- 1/\alpha } [  ,$ we conclude that 
$$  \E[ | I^1_t|] \le C  \E \int_0^t  \left( \left| a(\tilde X_s  ) - a(X_s )\right| +  | \tilde \mu_{s}(f))    -  \mu_{s}(f) | \right) ds .$$   
We now turn to the study of $ I_t^2.$  
We have
\begin{align*}
|I_t^{ 2}| \le \int_{[0,t]\times] \underline f^{- 1/\alpha },\infty[}\left| a\left (\tilde  X_{s-}+\tilde \mu_{s-}(f)^{1/\alpha}z\right)-a\left( X_{s-}+\mu_{s-}(f)^{1/\alpha}z\right) \right| M(ds,dz)\\
+  \int_{[0,t]\times] \underline f^{- 1/\alpha },\infty[} \left| a\left(\tilde X_{s-}\right)-a\left( X_{s-}\right)\right| M(ds,dz).
\end{align*}
Observe that the total mass of $M( [0,t]\times] \underline f^{- 1/\alpha },\infty[) $ is finite almost surely and that $ m ( ] \underline f^{- 1/\alpha },\infty[ ) < \infty,$ where $ ds m ( dz)$ denotes the intensity measure of $M ( ds, dz).$ Therefore, we only have to consider the first term in the above expression. We start from 
\begin{multline}\label{eq:multfirststep}
\left| a\left (\tilde  X_{s-}+ \tilde \mu_{s-}(f)^{1/\alpha}z\right)-a\left( X_{s-}+\mu_{s-}(f)^{1/\alpha}z\right) \right|   \\
\le \left| a\left (\tilde  X_{s-}+\tilde \mu_{s-}(f)^{1/\alpha}z\right)-a\left(\tilde  X_{s-}+\mu_{s-}(f)^{1/\alpha}z\right) \right| \\
 \quad \quad + 
\left| a\left (\tilde X_{s-}+ \mu_{s-}(f)^{1/\alpha}z\right)-a\left(X_{s-}+ \mu_{s-}(f)^{1/\alpha}z\right) \right| .
\end{multline}
Using again the concavity of $a$ and the fact that all arguments appearing in the function $a$ are positive, this is in turn upper bounded by 
$$
 \left| a\left (\tilde \mu_{s-}(f)^{1/\alpha}z\right)-a\left( \mu_{s-}(f)^{1/\alpha}z\right) \right|  + 
\left| a\left (\tilde  X_{s-}\right)-a\left( X_{s-}\right) \right|.
$$
Since the arguments $ \tilde \mu_{s-}(f)^{1/\alpha}x $ and $  \mu_{s-}(f)^{1/\alpha}x$ are both lower bounded by $1$ and $ a(x) = c + x^q $ for all $ x \geq 1,$ we deduce that 
\begin{multline}\label{eq:it2}
|I_t^{ 2} |\leq \int_{[0,t]\times] \underline f^{- 1/\alpha },\infty[} \left| a\left(\tilde X_{s-}\right)-a\left( X_{s-}\right)\right| M(ds,dz) \\
+ \int_{[0,t]\times [ \underline f^{-1/\alpha},\infty[}z^q\left |\tilde\mu_{s-}(f)^{q/\alpha}- \mu_{s-}(f)^{q/\alpha}     \right | M(ds,dz).
\end{multline}
Notice that the mapping $y\mapsto y^{q/\alpha}$ is Lipschitz on $ [ \underline f , \infty [.$ Using the fact that $ z^q  $ is integrable with respect to $ m ( dz) $  on $ [\underline f^{-1/\alpha}, \infty [ ,$  and the compensation formula for random measures, we finally obtain
\begin{equation*}
\E |I_t^{ 2}| \le C E \int_0^t   \left( \left| a(\tilde X_s  ) - a(X_s )\right| +  | \tilde \mu_{s}(f))    -  \mu_{s}(f) | \right) ds ,
\end{equation*}
which allows us to conclude.

We finally deal with the case when one of the two processes does not necessarily stay positive for all times. Suppose w.l.o.g. that it is $ \tilde X$ staying positive for all times. Coming back to \eqref{eq:multfirststep}, we may still upper bound 
$$ \left| a\left (\tilde  X_{s-}+\tilde \mu_{s-}(f)^{1/\alpha}z\right)-a\left(\tilde  X_{s-}+\mu_{s-}(f)^{1/\alpha}z\right) \right| \le \left| a\left (\tilde \mu_{s-}(f)^{1/\alpha}x\right)-a\left( \mu_{s-}(f)^{1/\alpha}x\right) \right|  ,$$
which is then treated as before. Moreover, using the second item of \eqref{eq:goodtoknow}, we have 
$$ \left| a\left (\tilde X_{s-}+ \mu_{s-}(f)^{1/\alpha}z\right)-a\left(X_{s-}+ \mu_{s-}(f)^{1/\alpha}z\right) \right| \le 2 a( 2 |\tilde X_{s-}-X_{s-}|) .$$ 
Since, thanks to the first item of \eqref{eq:goodtoknow},  $ \left| a\left (\tilde  X_{s-}\right)-a\left( X_{s-}\right) \right| \le 2 a ( |\tilde X_{s-}-X_{s-}|) ) \le 2 a( 2 |\tilde X_{s-}-X_{s-}|), $ ($a$ being non-decreasing), this allows to conclude. 
\end{proof}

\begin{rem}
In the above proof, we have used several times that jumps and the state space (of at least one of the processes) is positive, and that $a$ is concave on $ \R_+.$ Therefore, the above arguments do not apply in a general framework where particles might take negative values. 
\end{rem}

\subsection{Proof of Theorem \ref{theo:2main}}\label{sec:24}

Theorem \ref{theo:first} allows us to rewrite 

\begin{equation}\label{eq:difXY}
 X_t^{N, i } = Y^{N, i }_t + \tilde  R^{N, i }_t, 
 \end{equation} 
with 
\begin{multline*}
Y^{N, i }_t=  X^{N,i}_0 +   \int_0^t  b(X^{N, i}_s , \mu_s^N)  ds +  \int_{[0,t]\times\R_+} \psi(X^{N, i}_{s-})   \indiq_{ \{ z \le  f ( X^{N, i}_{s-}) \}} \bar \pi^i (ds,dz) \\
+\int_0^t (\mu^N_{ \tau_s} (f) )^{ 1 / \alpha } d S^{N, \alpha}_s .
\end{multline*}
Notice that while by construction $ X_t^{N, i} \geq 0, $ it might happen that $Y_t^{N, i } \le 0.$

{\bf Step 1.} Applying \eqref{eq:goodtoknow} to $ x := X_t^{N, i } $ and $y := Y_t^{N, i } ,$  we have that  
\begin{equation}\label{eq:difaXaY}
 | a ( X_t^{N, i } ) -  a ( Y_t^{N, i } ) |\le 2  a( | \tilde  R_t^{N, i }|)  \le C  d_q (| \tilde  R_t^{N, i }|) .
 \end{equation} 

{\bf Step 2.} We now apply It\^o's formula to $ a (Y_t^{N, i } ) $ and obtain 
\begin{multline}
a(Y^{N, i}_t) =  a(X^{N,i}_0) +   \int_0^t  a' ( Y^{N, i}_s) b(X^{N, i}_s , \mu_s^N)  ds \\
+  \int_{[0,t]\times\R_+}[  a(Y^{N, i}_{s-} +  \psi(X^{N, i}_{s-})) - a(Y^{N, i}_{s-}) ]   \indiq_{ \{ z \le  f ( X^{N, i}_{s-}) \}} \bar \pi^i (ds,dz) \\
+\int_{[0, t ]  \times \R_+} [a(Y^{N, i}_{s-} + (\mu^N_{ \tau_s} (f) )^{ 1 / \alpha }  x ) - a(Y^{N, i}_{s-}) ]   M  ( ds, dx)  .
\end{multline}

Using  \eqref {eq:difXY} this implies 
\begin{multline}
a(Y^{N, i}_t) =  a(X^{N,i}_0) +   \int_0^t  a' ( X^{N, i}_s) b(X^{N, i}_s , \mu_s^N)  ds \\
+  \int_{[0,t]\times\R_+}[  a(X^{N, i}_{s-} +  \psi(X^{N, i}_{s-})) - a(X^{N, i}_{s-}) ]   \indiq_{ \{ z \le  f ( X^{N, i}_{s-}) \}} \bar \pi^i (ds,dz) \\
+\int_{[0, t ]  \times \R_+} [a(X^{N, i}_{s-} + (\mu^N_{ s-} (f) )^{ 1 / \alpha }   x ) - a(X^{N, i}_{s-}) ]   M  ( ds, dx)   + B^N_t + \Psi^N_t + I^N_t,
\end{multline}
where $  B^N_t ,  \Psi^N_t,  I^N_t  $ are three error terms given by 
\begin{eqnarray*}
  B^N_t & = &  \int_0^t  [a' ( Y^{N, i}_s) -  a' ( X^{N, i}_s)] b(X^{N, i}_s , \mu_s^N)  ds ,\\
 \Psi^N_t &=&  \int_{[0,t]\times\R_+}[  (a(Y^{N, i}_{s-} +  \psi(X^{N, i}_{s-})) - a(Y^{N, i}_{s-})) - ( a(X^{N, i}_{s-} +  \psi(X^{N, i}_{s-})) - a(X^{N, i}_{s-}) )  ]  \\
 && \quad \quad \quad \quad  \indiq_{ \{ z \le  f ( X^{N, i}_{s-}) \}} \bar \pi^i (ds,dz) , \\
 I^N_t &=& I_t^N ( X^{N, i } , Y^{N, i }, (\mu^N_{t-})_t, (\mu^N_{\tau_t })_t) ,
 \end{eqnarray*}
where we used the notation of Lemma \ref{lem:it}. 

The first error term is controlled by 
$$|B^N_t| \leq  \int_0^t |  a' ( X^{N, i}_s) - a' ( Y^{N, i}_s) | | b(X^{N, i}_s , \mu_s^N)| ds\le C \int_0^t | a( X^{N, i}_s) - a( Y^{N, i}_s) |  ds , $$
since $b$ is bounded, due to the properties of the function $a.$  But using \eqref{eq:difaXaY}, 
$$ \int_0^t | a( X^{N, i}_s) - a( Y^{N, i}_s) |  ds \leq C \int_0^t d_q( | \tilde  R_s^{N,i } |)   ds .$$
  
Moreover, using \eqref{eq:samex2bis} of Lemma \ref{lem:it} and the fact that $ a( 2 | X^{N, i}_s -  Y^{N, i}_s|) \le C d_q(| \tilde  R_{s}^{N, i }|)  $ (where we have used that $ a (x) \le C d_q ( x) $ for any $ x \geq 0$), we obtain that 
$$ \E | I_t^N| \le C \int_0^t \E \left( d_q( | \tilde  R_{s}^{N, i }|) +   |\mu^N_{ \tau_s} (f) - \mu^N_{ s} (f)    |\right)  ds . $$ 

The second error  term $\Psi^N_t$ can be controlled similarly, since $f$ and $ \psi$ are bounded and $a-$Lipschitz. Resuming the above discussion, we obtain 
$$
\E \left( | B_t^N| + |\Psi_t^N| + |I_t^N|\right)  \le C  \int_0^t  \E \left( d_q (| \tilde R_s^{N,i } |)\right) ds + C \int_0^t \E \left( |\mu^N_{ \tau_s} (f) - \mu^N_{ s} (f)    | \right)ds .
$$
Finally, we use that, by exchangeability, for any fixed $i,$ 
$$ \E \left( |\mu^N_{ \tau_s} (f) - \mu^N_{ s} (f)    | \right) \le \E \left( | f ( X^{N, i }_{\tau_s } - f( X^{N, i }_{s }| \right) \le C N^{ 1 - q/\alpha } \delta, $$
where the last inequality is a consequence of Proposition 4.4 in \cite{DEE}. 
Gathering the errors of Steps {\bf 1.} and {\bf 2.} and recalling Theorem \ref{theo:first}, we then obtain the result. 
$\qed$

\section{A mean field version of the limit system}\label{sec:4}
Up to an error term, the representation obtained in \eqref{eq:axfinite} is a mean field version of the limit system, that is, of the system, 
given for any fixed $N,$ by 
\begin{multline}\label{eq:limitsystembis}
\tilde X^{N,i}_t =  X^{i}_0 +   \int_0^t  b(\tilde X^{N, i}_s , \tilde \mu^N_s )  ds +  \int_{[0,t]\times\R_+ } \psi (\tilde X^{ N,i}_{s-} )  \indiq_{ \{ z \le  f ( \tilde X^{N, i}_{s-}) \}} \bar \pi^i (ds,dz) \\
+ \int_{[0,t] } \left( \tilde \mu^N_{s-}(f)  \right)^{1/ \alpha}   d S_s^{N, \alpha} , 1 \le i \le N, 
\end{multline} 
where $ \tilde \mu_s^N = \frac1N \sum_{i=1}^N \delta_{\tilde X^{N,i}_s} .$  Comparing the finite system with this mean field system, we obtain
\begin{prop}\label{prop:51}
Grant Assumptions~\ref{ass:b}, ~\ref{ass:1} and \ref{ass:nu0bis}.
Then for all $t \le T,$ 
$$  \E \left( |a( X_t^{N, i } )- a(\tilde X_t^{N, i } )|  \right) \le C_T \left( N^{ 2(1 - q/\alpha)} \delta  + (N \delta)^{ - q/2} \delta^{ q/\alpha - 1 } + \delta^{q /\alpha}  
  + N^{ - q/\alpha} \right) .$$
\end{prop}

\begin{proof}
Ito's formula implies that 
\begin{multline}
a(\tilde X^{N, i}_t) =  a(X^{i}_0) +   \int_0^t  a' ( \tilde X^{N, i}_s) b(\tilde X^{N, i}_s , \tilde \mu_s^N)  ds \\
+  \int_{[0,t]\times\R_+}[  a(\tilde X^{N, i}_{s-} +  \psi(\tilde X^{N, i}_{s-})) - a(\tilde X^{N, i}_{s-}) ]   \indiq_{ \{ z \le  f ( \tilde X^{N, i}_{s-}) \}} \bar \pi^i (ds,dz) \\
+\int_{[0, t ]  \times \R_+} [a(\tilde X^{N, i}_{s-} + (\tilde \mu^N_{ s-} (f) )^{ 1 / \alpha }  x ) - a(\tilde X^{N, i}_{s-}) ]   M  ( ds, dx) ,
\end{multline}
where $M$ is the jump measure of $ S^{N, \alpha}.$

Since $ \frac1N \sum_{i=1}^N \delta_{ ( X_s^{N, j },  \tilde X_s^{N, j })} $ is an obvious coupling of $ \mu_s^N $ and of $ \tilde \mu_s^N, $
by the properties of $b$ and using that $| d_q (x) - d_q(y) |\le C |a(x) - a(y) | , $ 
$$ |b( X^{N, i}_s ,  \mu^N_s ) - b(\tilde X^{N, i}_s , \tilde \mu^N_s )| \le C ( | a(X^{N, i}_s) -a(\tilde X^{N, i}_s) | +\frac{1}{N} \sum_{j=1}^N | a(X^{N, j}_s) -a(\tilde X^{N, j}_s) |) .$$ 

Using the representation \eqref{eq:axfinite}, the properties of the function $a,$ the fact that $\psi $ is bounded and $a-$Lipschitz,  the notation of Lemma \ref{lem:it} and the above bound on the differences in the drift, this implies that,  for a convenient constant $C > 0, $
\begin{multline}
\frac1C  |a( X_t^{N, i } )- a(\tilde X_t^{N, i } )| \le  \int_0^t \left( | a( X_s^{N, i } )- a(\tilde X_s^{N, i } )| +\frac{1}{N} \sum_{j=1}^N | a(X^{N, j}_s) -a(\tilde X^{N, j}_s) | \right)ds \\
 + \int_{[0,t]\times\R_+} | a( X^{N, i}_{s-} ) - a(\tilde X^{N, i}_{s-}) |   \indiq_{ \{ z \le  \|f\|_\infty \}} \bar \pi^i (ds,dz) \\
 +  \int_{[0,t]\times\R_+}    \indiq_{ \{ f ( \tilde X^{N, i}_{s-}) \wedge f (  X^{N, i}_{s-}) <z \le  f ( \tilde X^{N, i}_{s-}) \vee f (  X^{N, i}_{s-})} \bar \pi^i (ds,dz)\\
 + I_t ( X^{N, i } , \tilde X^{N, i } , \mu^N , \tilde \mu ^N )  + |\bar R_t^{N, i}| ,
 \end{multline}
where $\bar R_t^{N, i}$ is as in Theorem \ref{theo:2main}.

Taking expectation, using the fact that $f$ is $a-$Lipschitz and the exchangeability of both systems, together with \eqref{eq:samex2} of Lemma \ref{lem:it}, this implies that 
$$ \E  |a( X_t^{N, i } )- a(\tilde X_t^{N, i } )|  \le C \int_0^t \E  |a( X_s^{N, i } )- a(\tilde X_s^{N, i } )|  ds + \E ( |\bar R_t^{N, i}|).  $$
Gronwall's lemma then implies the result. 
\end{proof}

To finally compare the auxiliary system with the limit system driven by $ S^{N, \alpha},$ we need to control the distance between $ \bar \mu_t^N = \frac{1}{N} \sum_{i=1}^N \delta_{ \bar X_t^{N, i } } ,$  the empirical measure of the limit system, and $\bar \mu_t = {\mathcal L} ( \bar X_t^{N, i } | S^{N, \alpha} ) .$ We quote the following result, which is based on Theorem 1 of \cite{fournierguillin}, from \cite{DEE}.

\begin{prop}\label{prop:fg}{(Proposition 5.6 of \cite{DEE}, Theorem 1 of \cite{fournierguillin})}
Grant Assumption \ref{ass:nu0bis}. Then we have that 
\begin{equation}\label{eq:wasserstein}
\E ( W_q ( \bar \mu_t^N, \bar \mu_t ) ) \le C_t (q, \alpha) 
 \left\{ 
\begin{array}{ll}
N^{ - 1/2}, & q > \frac12 \\
N^{- q }, & q < \frac12 
\end{array} 
\right\} .
\end{equation}
\end{prop}

We may now state our final auxiliary result that controls the distance between the auxiliary particle system and the limit system. 
 
\begin{prop}\label{prop:52}
Grant Assumptions ~\ref{ass:b}, ~\ref{ass:1} and \ref{ass:nu0bis}. Then for all $ t \le T,$ 
$$ \E \left(   | a(\bar X_t^{N, i })  - a( \tilde X_t^{N, i }) | \right) \le C_T (q, \alpha) [ N^{ - q}  \indiq_{\{ q < \frac12\}} + N^{- 1/2 }\indiq_{\{ q > \frac12\}}].$$
Moreover, if \eqref{eq:linearb} holds with bounded $ \tilde b,$ than we obtain the better control 
$$ \E \left(   | a(\bar X_t^{N, i })  - a( \tilde X_t^{N, i }) | \right) \le C_T N^{- 1/2}.$$
\end{prop} 

\begin{proof}
We rewrite for any $ 1 \le i \le N , $ 
\begin{multline}\label{eq:limitsystemtrois}
a(\bar X^{N,i}_t )=  a(X^{i}_0) +   \int_0^t a'(\bar X^{N, i}_s)  b(\bar X^{N, i}_s , \bar \mu^N_s )  ds\\
 +  \int_{[0,t]\times\R_+ } [ a (\bar X^{ N,i}_{s-} +\psi (\bar X^{ N,i}_{s-} )) -  a (\bar X^{ N,i}_{s-} )]  \indiq_{ \{ z \le  f ( \bar X^{N, i}_{s-}) \}} \bar \pi^i (ds,dz) \\
+ \int_{[0,t] \times \R_+ } \left[ a( \bar X^{ N,i}_{s-} +  (\bar \mu^N_{s-}(f) )^{1/ \alpha} x ) - a (\bar X^{ N,i}_{s-} )\right]  M(ds, dx)   + R_t^{N, 4 } , 
\end{multline} 
where $ \bar \mu_t^N = \frac{1}{N} \sum_{i=1}^N \delta_{ \bar X_t^{N, i } } $ is the empirical measure of the limit system and where 
\begin{multline}\label{eq:64}
 R_t^{N, 4 } = \int_0^t a'(\bar X^{N, i}_s) \left[ b(\bar X^{N, i}_s , \bar \mu_s ) - b(\bar X^{N, i}_s , \bar \mu^N_s ) \right] ds \\
 + \int_{[0,t] \times \R_+ } \left[ a( \bar X^{ N,i}_{s-} +  (\bar \mu_{s-}(f) )^{1/ \alpha} x ) - a( \bar X^{ N,i}_{s-} +  (\bar \mu^N_{s-}(f) )^{1/ \alpha} x )\right] M(ds, dx) .
\end{multline} 
We first work under Assumption \ref{ass:b} only and obtain, using the same arguments as in the proof of Proposition \ref{prop:51}, 
$$ \E \left(   | a(\bar X_t^{N, i })  - a( \tilde X_t^{N, i }) | \right)  \le C \int_0^t  \E \left(   | a(\bar X_s^{N, i })  - a( \tilde X_s^{N, i }) | \right) ds +  \E (| R_t^{N, 4 } |).$$
By the properties of the function $b,$ 
$$ |  b(\bar X^{N, i}_s , \bar \mu_s ) - b(\bar X^{N, i}_s , \bar \mu^N_s ) | \le C W_q ( \bar \mu_s, \bar \mu_s^N).$$
Moreover, using once more the result of Lemma \ref{lem:it},
\begin{multline*}
\E \left| \int_{[0,t] \times \R_+ } \left[ a( \bar X^{ N,i}_{s-} +  (\bar \mu_{s-}(f) )^{1/ \alpha} x ) - a( \bar X^{ N,i}_{s-} +  (\bar \mu^N_{s-}(f) )^{1/ \alpha} x )\right] M(ds, dx) \right| 
\\
\le 
C \E \int_0^t \left| \bar \mu_{s}(f) ) - \bar \mu^N_{s}(f) \right|  ds  .
\end{multline*}
Due to \eqref{eq:fq2}, we have that 
$$ \left| \bar \mu_{s}(f) ) - \bar \mu^N_{s}(f) \right| \le C W_q ( \bar \mu_s, \bar \mu_s^N ) .$$ 
Taking expectation, the conclusion then follows from 
$$ \E | R_t^{N, 4 } |  \le C \int_0^t  \E W_q ( \bar \mu_s^N, \bar \mu_s ) ds  ,$$
such that Proposition \ref{prop:fg} allows to conclude.

We finally work under the stronger Assumption that \eqref{eq:linearb} holds with bounded $ \tilde b.$  Going back to formula \eqref{eq:64}, we see that we have to control 
$$ |b ( \bar X^{N, i }_s , \bar \mu_s ) - b ( \bar X^{N, i }_s, \bar \mu^N_s)|.$$

Since $ b$ is linear in the measure variable, we obtain 
$$
b ( \bar X^{N, i }_s, \bar \mu^N_s ) - b ( \bar X^{N, i }_s, \bar \mu_s) =
  \frac1N \sum_{j=1}^N \left(  \tilde  b ( \bar X^{N, i }_s, \bar X^{N, j}_s)   - \int_\R  \tilde b ( \bar X^{N, i }_s, y  ) \bar \mu_s ( dy) \right)  .
$$
For any $j \neq i, $ conditioning on $S^\alpha $ and $ \bar X^{N, i }_s, $ the variables $Y^j_s :=  \tilde  b ( \bar X^{N, i }_s, \bar X^{N, j}_s )  - \int_\R  \tilde b ( \bar X^{N, i }_s, y  ) \bar \mu_s ( dy) $ are i.i.d., centered and bounded. 
The usual variance estimate for sums of i.i.d. variables implies then that
$$ \E \left( | \frac1N \sum_{j \neq i } Y^j_s | \; \big| S^\alpha , \bar X^{N, i }_s  \right) \le C N^{ - 1/2}.$$  
Since $ | Y^i_s | /N $ (the diagonal term) is bounded by $ C/N, $ integrating with respect to $ \bar X^{N, i }_s $ and $ S^\alpha $ then gives the upper bound 
$$ \E \left(  |b ( \bar X^{N, i }_s , \bar \mu_s ) - b ( \bar X^{N, i }_s, \bar \mu^N_s)| \right) \le C N^{ - 1/2}.$$
The expression $  \left| \bar \mu_{s}(f) ) - \bar \mu^N_{s}(f) \right| $ is controlled analogously. This concludes the proof. 
\end{proof}
We are now able to finalize the 

\begin{proof}[Proof of Theorem \ref{strongconvergence}]
Putting together Theorem \ref{theo:2main}, Proposition \ref{prop:51} and Proposition \ref{prop:52},   we obtain for all $ t \le T,$ 
\begin{multline*}
\E ( | a(X^{N,1}_t) -a(\bar X^{N, 1}_t)  | ) \leq \\
 C_T (q, \alpha)  \left( N^{ 2(1 - q/\alpha)} \delta + (N \delta)^{ - q/2} \delta^{ q/\alpha - 1 } + \delta^{q/\alpha}    + N^{ - q/\alpha} + [ N^{ - q }  \indiq_{\{ q < \frac12\}} + N^{- 1/2 }\indiq_{\{ q > \frac12\}}] \right) ,
 \end{multline*}
which holds for any $ q < \alpha.$ Since $ N \delta \to \infty, $ clearly, $ N^{ - q/\alpha} \le C \delta^{q/\alpha} $ for $N $ sufficiently large, such that we do not need to care of the term $N^{ - q/\alpha} .$ 

We now choose $ \delta  $ such that 
$$ N^{ 2(1 - q/\alpha)} \delta  =  (N \delta)^{ - q/2 } \delta^{ q/\alpha - 1 };$$
that is, 
\begin{equation}\label{eq:delta} \delta = N^{-  \frac{ 4 (1 - \frac{q}{\alpha}) + q }{ 2 ( 1 - \frac{q}{\alpha} ) + q + 2  }} .
\end{equation}
For this choice of $ \delta $ we then obtain that 
$$ N^{ 2(1 - q/\alpha)} \delta + (N \delta)^{ - q/2 } \delta^{ q/\alpha - 1 } = 2 N^{ 2 ( 1 - q/\alpha) -  \frac{ 4 (1 - \frac{q}{\alpha}) + q }{ 2 ( 1 - \frac{q}{\alpha} ) + q + 2  } }.$$
Remembering that by item vi) of Proposition \ref{prop:a}, $ \E ( | a(X^{N,1}_t) -a(\bar X^{N, 1}_t)  | ) \geq c_1 \E ( | d_q(X^{N,1}_t) -d_q(\bar X^{N, 1}_t)  | ) ,$ this  concludes the proof.  
\end{proof}

\section{Appendix}

We first give the 
\begin{proof}[Proof of Lemma \ref{lem:useful}]
We first prove the inequality for $ z=0.$ Let us write $ r= x-y.$ When $y \geq 0 $ and $ r \geq 0, $ we may use the fact that $a$ is sublinear on $ \R_+$ and obtain $ | a(x) - a (y) |  = | a ( y+r) - a(y) |  \le a (r) .$ When $ y \geq 0 $ but $ r < 0,$ since $a$ is increasing on $ \R_+,$ we have $ | a(x) - a (y) | = a( y) - a(x) = a(x -r) - a(x) \le a ( -r)  = a ( |r|) . $

Suppose now that $ y < 0.$ Then $x = y +r \le r $ and $ r \geq 0, $ such that $ x \le r$ and $ |y| = r - x \le  r.$ In this case, $| a(x) - a (y) |  \le a (x) + | a(y) |= a (x) + a ( -y) \le 2 a (  r ). $

We now turn to the proof of the general statement with $z > 0.$ Firstly, if $ y \geq 0, $ then we have $ |a(y+z) - a(x+z) | \le | a (x) - a (y) |   \le 2 a ( |x-y|)  $ according to our preliminary step. Suppose now that $ y < 0 $ but $ y + z \geq 0.$ Then $ r \geq 0, $ and we have $  |a(y+z) - a(x+z) | =a(x+z) - a(y+z) =  a ( y+ z + r ) - a(y+z) \le a (r) .$ Finally, if $ y+z < 0$ and $y < 0, $ then $ z \le | y | \le  r $ and $ |a ( y+z )| = a(|y|- z ) \le a( |y|) $ (since $a$ is increasing on $ \R_+$), which is in turn bounded by $ a (  r ) $ Moreover, since we also have that $ x \le r, $  $ a ( x+z ) \le a ( 2 r )  ,$ and this concludes the proof.
\end{proof}

We finally give the 
\begin{proof}[Proof of Remark \ref{cor:couplingboundb}]
It suffices to provide an upper bound for 
$$ b(\tilde x, \mu ) - b( \tilde x, \tilde \mu) = \int_0^1 \int_{\R_+} \delta_\mu b ( \tilde x, y, (1 - \lambda ) \tilde \mu + \lambda \mu ) (\mu - \tilde \mu ) (dy ) d \lambda .$$ 
Let $\pi ( dy, dy') $ be any coupling of $ \mu $ and $ \tilde \mu .$ Then 
\begin{multline*}
\int_0^1 \int_{\R_{+}} \delta_\mu b ( \tilde x, y, (1 - \lambda ) \tilde \mu + \lambda \mu ) (\mu - \tilde \mu ) (dy ) d \lambda \\
= \int_0^1 \int_{ \R_+\times \R_+ } \pi (dy, dy') [ \delta_\mu b ( \tilde x, y, (1 - \lambda ) \tilde \mu + \lambda \mu )  - \delta_\mu b ( \tilde x, y', (1 - \lambda ) \tilde \mu + \lambda \mu ) ]d \lambda .
\end{multline*}
By \eqref{eq:boundb}, 
$$ | \delta_\mu b ( \tilde x, y, (1 - \lambda ) \tilde \mu + \lambda \mu )  - \delta_\mu b ( \tilde x, y', (1 - \lambda ) \tilde \mu + \lambda \mu ) | \le C |a(y) - a(y') | ,$$
such that 
\begin{equation}\label{eq:firstupperboundcoupling}
\left| \int_0^1 \int_\R \delta_\mu b ( \tilde x, y, (1 - \lambda ) \tilde \mu + \lambda \mu ) (\mu - \tilde \mu ) (dy )d \lambda \right|  
\le C \int_{ \R_+\times \R_+ } \pi (dy, dy') |a(y) - a(y') | .
\end{equation}
Since the above upper bound holds for any coupling, the assertion follows. 
\end{proof}

\begin{acks}[Acknowledgments]
We thank Nicolas Fournier and Elisa Marini for illuminating comments on an early version of this paper. 

The authors acknowledge support of the Institut Henri Poincar\'e (UAR 839 CNRS-Sorbonne Universit\'e), and LabEx CARMIN (ANR-10-LABX-59-01). This work has been conducted as part of  the ANR project ANR-19-CE40-0024.
\end{acks}


\begin{thebibliography}{10}

\bibitem{andreis_mckeanvlasov_2018}
	{\sc Andreis, L., Dai~Pra, P., Fischer, M.}
	\newblock{{McKean}{Vlasov} limit for interacting systems with simultaneous jumps.}
	\newblock {\em Stochastic Analysis and Applications} {36}(6), 960--995 (2018).


\bibitem{carmona_mean_2016}
{\sc Ren{\'e} Carmona, Fran{\c c}ois Delarue, and Daniel Lacker}
\newblock{Mean field games with common noise}
\newblock {\em The Annals of Probability} {44} (6), 3740--3803, 2016. 

\bibitem{cavallazzi}
{\sc Thomas Cavallazzi}
\newblock{Quantitative weak propagation of chaos for stable-driven McKean-Vlasov SDEs}
\newblock{\em ArXiv} https://arxiv.org/abs/2212.01079, 2022.




\bibitem{coghi_propagation_2016}
{\sc Michele Coghi and Franco Flandoli}
\newblock{Propagation of chaos for interacting particles subject to environmental noise}
\newblock {\em The Annals of Applied Probability} 26 (3), 1407--1442, 2016. 

\bibitem{DGLP}
{\sc Anna De~Masi and Antonio Galves and Eva L{\"o}cherbach and Errico Presutti}
\newblock {Hydrodynamic {Limit} for {Interacting} {Neurons}}
\newblock {\em Journal of Statistical Physics} 158, 866--902, 2015.

\bibitem{dermoune_propagation_2003}
{\sc Azzouz Dermoune}
\newblock{Propagation and conditional propagation of chaos for pressureless gas equations}
\newblock{\em Probability Theory and Related Fields} {126} (4),  459--476, 2003.

\bibitem{ELL1}
{\sc Xavier Erny and Eva  L{\"o}cherbach and Dasha
  Loukianova}
\newblock{Conditional propagation of chaos for mean field systems of interacting
  neurons}.
\newblock{\em Electronic J. Probab} 26, 1-25, 2021.

\bibitem{ELL2}
{\sc Xavier Erny and Eva  L{\"o}cherbach and Dasha
  Loukianova}
\newblock{Strong error bounds for the convergence to its mean field limit  for systems of interacting neurons in a diffusive scaling}.
\newblock{\em To appear in The Annals of Applied Probability}, 2023.


\bibitem{fournierguillin}
{\sc Nicolas Fournier and Arnaud Guillin}
\newblock {On the rate of convergence in Wasserstein distance of the empirical measure.}
\newblock {\em Probab. Theory Relat. Fields} 162, 707--738, 2015.


\bibitem{Jourdain}
{\sc Benjamin Jourdain and Alvin Tse}
\newblock {{Central limit theorem over non-linear functionals of empirical measures with applications to the mean-field fluctuation of interacting diffusions.}}
\newblock {\em Electronic Journal of Probability} 26, 1--34, 2021.

\bibitem{jourdainetal}
{\sc Benjamin Jourdain, Sylvie M\'el\'eard, and Wojbor Woyczynski}
\newblock{Nonlinear SDEs driven by L\'evy processes and related PDEs}
\newblock{\em ALEA Lat. Am. J. Probab. Math. Stat.} 4,1--29, 2007.

\bibitem{graham92}
{\sc Carl Graham} 
\newblock{{McKean}-{Vlasov} {Ito}-{Skorohod} equations, and nonlinear
  diffusions with discrete jump sets}
\newblock{\em Stochastic Processes and their Applications} {40} (1), 69--82, 1992.

\bibitem{DEE}
{\sc  Eva  L{\"o}cherbach and Dasha
  Loukianova and Elisa Marini}
\newblock{Strong propagation of chaos for systems of interacting particles with nearly stable jumps}.
\newblock{\em arXiv}, 2024.


\end{thebibliography}

\end{document}